\documentclass[12pt]{amsart}
\usepackage{amssymb}
\usepackage{hyperref,hyphenat}
\usepackage[all]{xy}
\usepackage[toc,page,title,titletoc,header]{appendix}
\usepackage{xcolor}

\numberwithin{equation}{section}

\theoremstyle{plain}
\newtheorem{Thm}{Theorem}[section]
\newtheorem{Prop}[Thm]{Proposition}
\newtheorem{Lem}[Thm]{Lemma}

\newtheorem{Cor}[Thm]{Corollary}
\theoremstyle{remark}
\newtheorem{Rem}[Thm]{Remark}
\theoremstyle{definition}
\newtheorem{Def}[Thm]{Definition}

\def\A{\mathbb A}
\def\G{\mathbb G}
\def\Z{\mathbb Z}
\def\Q{\mathbb Q}
\def\R{\mathbb R}
\def\C{\mathbb C}
\def\P{\mathbb P}
\def\sM{\mathcal M}
\def\sO{\mathcal O}
\def\Stab{{\rm Stab}}
\def\tors{{\rm tors}}
\def\PrePer{{\rm PrePer}}
\def\GL{{\rm GL}}
\def\Spec{{\rm Spec}}
\def\Zar{{\rm Zar}}
\def\Per{{\rm Per}}
\def\diag{{\rm diag}}
\def\Gal{{\rm Gal}}
\def\Pic{{\rm Pic}}
\def\adj{{\rm adj}}

\def\excp{{\rm excp}}
\def\la{\lambda}
\def\La{\Lambda}

\begin{document}
\title[]{Cyclotomic integral points for affine dynamics}

\author{Zhuchao Ji}
\address{Institute for Theoretical Sciences, Westlake University, Hangzhou 310030, China}
\email{jizhuchao@westlake.edu.cn}

\author{Junyi Xie}
\address{Beijing International Center for Mathematical Research, Peking University, Beijing 100871, China}
\email{xiejunyi@bicmr.pku.edu.cn}

\author{Geng-Rui Zhang}
\address{School of Mathematical Sciences, Peking University, Beijing 100871, China}
\email{grzhang@stu.pku.edu.cn, chibasei@163.com}

\subjclass[2020]{Primary 37P05; Secondary 37P15, 37P35}

\keywords{Regular endomorphisms, H\'enon maps, preperiodic points, maximal cyclotomic extension, cohomological hyperbolicity.}

\begin{abstract}
	Let $f:\A^N\to\A^N$ be a regular endomorphism of algebraic degree $d\geq2$ (i.e., $f$ extends to an endomorphism on $\P^N$ of algebraic degree $d$) defined over a number field. We prove that if the set of cyclotomic $f$-preperiodic points is Zariski-dense in $\A^N$, then some iterate $f^{\circ l}$ ($l\geq1$) is a quotient of a surjective algebraic group endomorphism $g:\G_m^N\to\G_m^N$, over $\overline{\Q}$. This result generalizes a theorem of Dvornicich and Zannier on cyclotomic preperiodic points of one-variable polynomials to higher dimensions. In fact, we prove a much more general rigidity result for dominant endomorphisms $f$ on an affine variety $X$ defined over a number field, concerning ``almost $f$-invariant'' Zariski-dense subsets of cyclotomic integral points. We apply our results to backward orbits of regular endomorphisms on $\A^N$ of algebraic degree $d\geq2$, and to periodic points of automorphisms of H\'enon type on $\A^N$.
\end{abstract}

\maketitle
\tableofcontents

\section{Introduction}
\subsection{Statement of the main results}
We establish rigidity results for algebraic dynamical systems of monomial type, defined as follows:
\begin{Def}
	Let $X$ be a quasi-projective variety of dimension $d$ and $f:X\dashrightarrow X$ a dominant rational self-map, both defined over a field $k$ of characteristic zero.
	\begin{enumerate}
		\renewcommand{\theenumi}{\roman{enumi}}
		\renewcommand{\labelenumi}{(\theenumi)}
		\item \label{monomialtype} We say that $(X,f)$ is of \emph{monomial type} if there exist integers $l\geq1$ and $n\geq d$, a group endomorphism $g:\G_{m,\overline{k}}^n\to\G_{m,\overline{k}}^n$ over $\overline{k}$, and a dominant morphism $\phi:\G_{m,\overline{k}}^n\to X_{\overline{k}}$ over $\overline{k}$ such that $f_{\overline{k}}^{\circ l}\circ\phi=\phi\circ g$, where $f_{\overline{k}}:X_{\overline{k}}\dashrightarrow X_{\overline{k}}$ is the base change of $f$ to $X_{\overline{k}}$;
		\item We say that $(X,f)$ is of \emph{strongly monomial type} if it is of monomial type and we can take $n=\dim(X)$ in the definition \eqref{monomialtype}.
	\end{enumerate}
\end{Def}
We fix an algebraic closure $\overline{\Q}$ of $\Q$ in $\C$. For an algebraic number $\alpha\in\overline{\Q}$, we define its \emph{house} to be $$C(\alpha)=\max\{|\sigma(\alpha)|\}_{\sigma:\Q(\alpha)\hookrightarrow\C}.$$
For a number field $K$, we define its \emph{maximal cyclotomic extension} as
$$K^c=K(U(\C)),$$
the subfield of $\C$ generated over $K$ by the group $U(\C)$ of all roots of unity in $\C$.

Our main result is the following theorem:
\begin{Thm}\label{thmmain}
Let $X\subseteq \A^N$ be a geometrically irreducible affine variety of dimension $d\geq1$ and $f:X\to X$ a dominant endomorphism, both defined over a number field $K$. Let $(x_1,\dots,x_N)$ denote the coordinates of $\A^N$. Assume that $P$ is a subset of $X(K^c)$ satisfying the following conditions:
\begin{itemize}
	\item (DCI, dense cyclotomic integral points) $P$ is Zariski-dense in $X$, and there exists an integer $M\geq1$ such that for every $y\in P$ and $1\leq i\leq N$, the multiple $M\cdot x_i(y)\in K^c$ is an algebraic integer;
	\item (BH, bounded house) there exists a constant $c\in\R_{>0}$ such that
	$$C(y):=\max\{C(y_i):1\leq i\leq N\}\leq c$$
	for every $y=(y_1,\dots,y_N)\in P$;
	\item (AI, almost invariant) $P\setminus f^{-1}(P)$ is not Zariski-dense in $X$.
\end{itemize}
Then $(X,f)$ is of monomial type (over $\overline{\Q}$).
\end{Thm}
Note that we only need the condition (AI) to hold for the system $(X,f^{\circ l})$ (i.e., $P\setminus (f^{\circ l})^{-1}(P)$ is not Zariski-dense in $X$) for some integer $l\geq1$.

We recall the notions of dynamical degrees and cohomological hyperbolicity. Let $X$ be a quasi-projective variety of dimension $d$ over a field $k$ of characteristic zero, and let $f:X\dashrightarrow X$ be a dominant rational self-map of $X$. Fix a projective compactification $X^\prime$ of $X$ and view $f:X^\prime\dashrightarrow X^\prime$ as a dominant rational self-map of $X^\prime$. Let $L$ be a nef and big line bundle in $\Pic(X^\prime)$. Let $\Gamma$ be the graph of $f$ in $X^\prime\times X^\prime$ (i.e., the Zariski-closure of $\{(x,f(x)):x\text{ is a closed point in }X^\prime\setminus I(f)\}$), and let $\pi_j:\Gamma\to X^\prime$ be the $j$-th projection for $j=1,2$. For $0\leq i\leq d$, the \emph{$i$-th degree} of $f$ (relative to $L$) is
\begin{equation}\label{degi}
	\deg_{i,L}(f):=\left((\pi_2^* L)^i\cdot(\pi_1^* L)^{d-i}\right).
\end{equation}
The \emph{$i$-th dynamical degree} of $f$ is
$$\la_i(f):=\lim_{n\to\infty}\deg_{i,L}(f^{\circ n})^{1/n}\geq1.$$
The above limit exists and is independent of the choices of $X^\prime$ and $L$; see \cite{Dang2020,Truong2020}. As in \cite{recursive}, we define $\mu_i(f):=\la_i(f)/\la_{i-1}(f)$ for $1\leq i\leq d$ and $\mu_{d+1}(f):=0$, called the \emph{cohomological Lyapunov multipliers} of $f$. The log-concavity of $(\la_i(f))_{i=0}^d$ \cite[Theorem 1.1 (3)]{Truong2020} shows that $(\mu_i(f))_{i=1}^{d+1}$ is non-increasing in $i$. We say $f$ is \emph{cohomologically hyperbolic} if $\mu_i(f)\neq1$ for every $1\leq i\leq d+1$, equivalently, if there is a unique $0\leq i\leq d$ such that $\la_i(f)=\max\{\la_j(f):0\leq j\leq d\}$.

We show that for cohomologically hyperbolic systems, being of monomial type is equivalent to being of strongly monomial type.
\begin{Thm}\label{chdimequal}
	Let $X$ be a quasi-projective variety of dimension $d$ and $f:X\dashrightarrow X$ a dominant rational self-map, both defined over a field $k$ of characteristic zero. Assume that $f$ is cohomologically hyperbolic. Then $(X,f)$ is of monomial type if and only if $(X,f)$ is of strongly monomial type.
\end{Thm}
For cohomologically hyperbolic endomorphisms on affine varieties over $\overline{\Q}$, by Theorem \ref{chdimequal} we can deduce that being of monomial type is equivalent to the conditions (DCI), (BH), and (AI) ``up to iterates and changing the base number field'':
\begin{Cor}\label{chequiv}
	Let $X$ be an irreducible affine variety and $f:X\to X$ a dominant endomorphism, both over $\overline{\Q}$. Assume that $f$ is cohomologically hyperbolic. Then the following statements are equivalent:
	\begin{enumerate}
		\item \label{chequiv1} $(X,f)$ is of monomial type;
		\item \label{chequiv2} $(X,f)$ is of strongly monomial type;
		\item \label{chequiv3} there exist a number field $K$ such that $(X,f)$ is defined over $K$, an integer $l\geq1$, and a subset $P\subseteq X(K^c)$ that satisfies (DCI), (BH), and (AI) for the system $(X,f^{\circ l})$.
	\end{enumerate}
\end{Cor}
The proofs of Theorem \ref{thmmain}, Theorem \ref{chdimequal}, and Corollary \ref{chequiv} will be given in \S \ref{seccycmainpf}.

\subsection{Applications}
We present some examples of endomorphisms on $\A^N$ to which our main theorem applies.

Let $N\in\Z_{>0}$ and $f:\A^N\to\A^N$ be a polynomial endomorphism over $\C$. We view $\A^N=\P^N\setminus\{z_0=0\}\subset\P^N$, where $[z_0,z_1,\dots,z_N]$ are the coordinates of $\P^N$. Then $f$ extends to a rational self-map $f:\P^N\dashrightarrow\P^N$. Let $I(f)\subset\P^N$ denote the indeterminacy locus of $f$ in $\P^N$. The algebraic degree of $f$ is $\deg_1(f):=\deg_{1,\sO_{\P^N}(1)}(f)$; see \eqref{degi}. If we write $f=(f_1,\dots,f_N):\A^N\to\A^N$ with each $f_i\in\C[z_1,\dots,z_N]$, then $\deg_1(f)=\max\{\deg(f_i):1\leq i\leq N\}$.

\subsubsection*{Preperiodic points of regular endomorphisms on affine spaces}
Our main theorems can be applied to regular endomorphisms on $\A^N$ of algebraic degree $d\geq2$.

For $N,d\in\Z_{>0}$, a polynomial endomorphism $f:\A^N\to\A^N$ over a field $k$ is called a regular endomorphism of algebraic degree $d$ if it extends to an endomorphism $f:\P^N\to\P^N$ of algebraic degree $d$. Write $f=(f_1,\dots,f_N)$ with each $f_j\in k[z_1,\dots,z_N]$. Then $f:\A^N\to\A^N$ is a regular endomorphism of algebraic degree $d$ if and only if each $f_j$ has degree $d$ and $f_h^{-1}(0)=\{0\}$ in $\A^N(\overline{k})=\overline{k}^N$, where $f_j^+$ is the sum of monomials of degree $d$ of $f_j$ ($1\leq j\leq N$) and $f_h=(f_1^+,\dots,f_N^+):\A^N\to\A^N$ is the homogeneous part of $f$. If $N=1$, then all (non-constant) polynomial endomorphisms on $\A^1$ are regular endomorphisms.
\begin{Thm}\label{appaff}
	Let $N\in\Z_{>0}$ and $f:\A^N\to\A^N$ be a regular endomorphism of algebraic degree $d\geq2$ defined over a number field $K$. Let $P:=\PrePer(f,\A^N(K^c))$ be the set of $K^c$-rational $f$-preperiodic points in $\A^N$. If $P$ is Zariski-dense in $\A^N$, then $(\A^N,f)$ is of strongly monomial type.
\end{Thm}

\subsubsection*{Backward orbits of regular endomorphisms on affine spaces}
We show that for cyclotomic points in a backward orbit of a regular endomorphism $f$ on $\A^N$ of algebraic degree $d\geq2$, the conditions (DCI), (BH), and (AI) have a simple equivalent form. Note that the backward orbit
$$\{z\in \A^N(\overline{\Q}):\exists n\geq1, f^{\circ n }(z)=x\}$$
of some point $x\in\A^N(\overline{\Q})$ may not be Zariski-dense in $\A^N$, while the set of $\overline{\Q}$-rational $f$-preperiodic points is always Zariski-dense in $\A^N$ because $f$ is polarized \cite{Fak03}. For $x$ in a non-empty Zariski open subset of $\A^N$, the backward orbit of $x$ is Zariski-dense \cite[Theorem 1.47]{DS10}.
\begin{Thm}\label{back}
	Let $N\in\Z_{>0}$ and $f:\A^N\to\A^N$ be a regular endomorphism of algebraic degree $d\geq2$ defined over a number field $K$. Let $x\in \A^N(\overline{K})$ and
	$$P=\{z\in \A^N(K^c):\exists n\geq1, f^{\circ n }(z)=x\}$$
	be the set of $K^c$-rational points in the backward orbit of $x$ under $f$. Then $P$ satisfies the conditions (DCI), (BH), and (AI) if and only if $P$ is Zariski-dense in $\A^N$. In this case, $f$ is of strongly monomial type.
\end{Thm}

\subsubsection*{Periodic points of automorphisms of H\'enon type on $\A^N$}
We can apply our results to periodic points of automorphisms of H\'enon type on $\A^N$, as defined below:
\begin{Def}
	For $N\in\Z_{\geq2}$ and a polynomial automorphism $f:\A^N\to\A^N$ defined over $\C$, we say that $f$ is of H\'enon type if $\deg_1(f)\geq2$ and $I(f)\cap I(f^{-1})=\emptyset$.
\end{Def}
We require $N\geq2$ because every automorphism $f:\A^1\to\A^1$ has (algebraic) degree one. Note that an automorphism $f:\A^N\to\A^N$ of H\'enon type is never a regular endomorphism. 

\begin{Thm}\label{Henon} Let $N\in\Z_{\geq2}$ and $f:\A^N\to\A^N$ be a polynomial automorphism of H\'enon type defined over a number field $K$. Let $P=\Per(f,\A^N(K^c))$ be the set of $K^c$-rational $f$-periodic points in $\A^N$. Then $P$ is not Zariski-dense in $\A^N$.
\end{Thm}
Assume that $N=2$. Let $f:\A^2\to\A^2$ be a polynomial automorphism defined over $\overline{\Q}$ such that $\la_1(f)>1$ (or equivalently, $f$ has positive entropy). By \cite{FM89}, after conjugation over $\overline{\Q}$, $f$ is of the form $f=f_1\circ\cdots\circ f_m$, where $m\in\Z_{>0}$ and for each $1\leq i\leq m$, $f_i(x,y)=(p_i(x)-a_i y,b_ix)$ with $a_i,b_i\in \overline{\Q}^*$ and $p_i(x)\in\overline{\Q}[x]$ of degree $\geq2$. It is clear that $f=f_1\circ\cdots\circ f_m$ is of H\'enon type. If $m=1$ and $b_1=1$, then $f(x,y)=(p_1(x)-a_1y,x)$ is called a H\'enon map. We immediately deduce the following corollary from Theorem \ref{Henon}:
\begin{Cor}\label{HenonA2} Let $f:\A^2\to\A^2$ be a polynomial automorphism with $\la_1(f)>1$ defined over a number field $K$. Then $\Per(f,\A^2(K^c))$ is not Zariski-dense in $\A^2$.
\end{Cor}
The philosophy behind these applications lies in the spirit of unlikely intersection problems \cite{UnlikelyBook}. For Theorem \ref{appaff} and Theorem \ref{Henon}, the set of cyclotomic $f$-preperiodic points (``special points'') should not be Zariski-dense in the underlying variety $X$, unless $X$ is a ``special variety'', i.e., $(X,f)$ is of strongly monomial type. For Theorem \ref{back}, we view cyclotomic points in a given backward orbit as ``special points''.

\medskip

The proofs of these applications will be given in \S \ref{SApp}.

\subsection{Motivation and previous results}
In 2007, Dvornicich and Zannier proved the following rigidity result \cite[Theorem 2]{DZ07} for one-variable polynomials with infinitely many cyclotomic preperiodic points:
\begin{Thm}[Dvornicich-Zannier]\label{thmDZ}
Let $f\in K[z]$ be a one-variable polynomial of degree $d\geq2$ over a number field $K$. Let $K^c$ be the maximal cyclotomic extension of $K$. Let $P:=\PrePer(f,K^c)$ be the set of $f$-preperiodic points in $K^c$. Assume that $P$ is an infinite set. Then there exists a polynomial $L\in\overline{\Q}[z]$ of degree $1$ such that $(L\circ f\circ L^{-1})(z)$ is either $z^d$ or $\pm T_d(z)$, where $T_d(z)$ is the Chebyshev polynomial of degree $d$.
\end{Thm}
\begin{Rem}\label{rmkgenDZ}
	 Observe that $z^d$ and $-z^d$ are always (affine) conjugate, and that $T_d(-z)=(-1)^d T_d(z)$ and $-T_d(z)$ are always (affine) conjugate. Hence the conclusion is equivalent to saying that $f$ is affine conjugate to $(\varepsilon z)^d$ or $T_d(\varepsilon z)$ over $\overline{\Q}$ for some $\varepsilon\in\{\pm1\}$, which is the original statement in \cite{DZ07}.
	 
	 We view $f$ as an endomorphism on $X=\A^1$. Since $P$ is infinite, it is Zariski-dense in $\A^1$. We conclude by Theorem \ref{appaff} that $f$ is of strongly monomial type. Then there exist $n\in\Z$, $l\in\Z_{>0}$, and $h\in\overline{\Q}(z)\setminus\overline{\Q}$ such that $h\circ z^n=f^{\circ l}\circ h$, and we have $n=\pm d^l$. It is well-known that then the polynomial $f$ must be conjugate to $z^d$ or $\pm T_d(z)$ (see, for example, \cite[Lemma 3.8]{milnor2006lattes} and \cite[Proposition 6.3(a) and Theorem 6.9(b)]{241}). We thus recover the theorem of Dvornicich and Zannier from Theorem \ref{appaff}.
	 
	 Therefore, our Theorem \ref{appaff} can be viewed as a higher-dimensional generalization of Theorem \ref{thmDZ}.
\end{Rem}
For $d\geq2$, the power map $z^d$ on $\A^1$ restricts to the group endomorphism $z^d$ on $\G_m^1\subseteq\A^1$. The polynomial $\pm T_d(z)$ on $\A^1$ is the quotient of $\pm z^d$ (which is the translation of the group endomorphism $z^d$ on $\G_m^1$ by the torsion point $\pm1$) on $\G_m^1$ by the automorphism $z\mapsto z^{-1}$ on $\G_m^1$, via the isomorphism
$$\G_m^1/\{z=z^{-1}\}\xrightarrow{\sim}\A^1,z\mapsto z+z^{-1},$$
see \cite[\S 6.2]{241}. Hence $z^d$ and $\pm T_d(z)$ are of strongly monomial type. Combined with Remark \ref{rmkgenDZ}, we conclude that for a polynomial $f\in\C[z]$ of degree $d\geq 2$, $f$ is of monomial type if and only if $f$ is of strongly monomial type if and only if $f$ is conjugate to $z^d$ or $\pm T_d(z)$.

\medskip

After \cite{DZ07}, Ostafe \cite{Ostafe} and Chen \cite{Chen} studied cyclotomic points in backward orbits for rational maps on $\P^1$ with a periodic critical point. Note that for a one-variable polynomial, $\infty$ is a fixed critical point. Ostafe's result \cite{Ostafe} in particular implies that if $f$ is a rational map on $\P^1$ (of degree $d\geq2$) with a periodic critical point defined over a number field $K$, and there exists $x\in \P^1(K)$ such that there are infinitely many cyclotomic points in the backward orbit of $x$, then $f$ is conjugate to $z^d$ or $\pm T_d(z)$. See also \cite[Lemma 3.4]{Back}. Our Theorem \ref{back} partially generalizes Ostafe's result to higher dimensions. 

When $K=\Q$, any abelian extension of $\Q$ is contained within some cyclotomic field. Thus, our Theorems \ref{appaff}, \ref{back}, \ref{Henon}, and \ref{HenonA2} lead to results on the distribution of abelian points for higher-dimensional dynamical systems when $K=\Q$. For results about abelian points in backward orbits of rational maps on $\P^1$, see \cite{AP20,Back,LP25,Leung}.

\subsection{Sketch of the proof}
We now sketch the proof of Theorem \ref{thmmain}. The proof is divided into four steps. We make a base change and work over $\overline{\Q}=\overline{K}$.

In the first step, by a generalization of a theorem of Loxton proved by Dvornicich and Zannier (Theorem \ref{thmL}), the conditions (DCI) and (BH) imply that there exist an integer $b\geq1$ and a finite subset $E\subset K$ containing $0$ such that $P\subseteq (\sum_{i=1}^b E\cdot U(\C))^N$ in $\A^N(\C)$. For each $a=(a_{ij})_{1\leq i\leq N,1\leq j\leq b}\in E^{bN}$, we define a morphism
$$\phi_a:G_a=\G_m^{bN}\to\A^N,(z_{ij})_{i,j}\mapsto\left(\sum_{j=1}^b a_{ij}z_{ij}\right)_i,$$
a subset $\La_a=\phi_a^{-1}(P)\cap G_a(\overline{K})_\tors$ of torsion points, and a closed subset $Z_a=\overline{\La_a}^\Zar\subseteq G_a$. Note that $P=\bigcup_{a}\phi_a(\La_a)\subseteq\bigcup_{a}\phi_a(Z_a)\subseteq X$ is Zariski-dense in $X$. We shrink $P$ suitably and only consider $a$ in $M:=\{a\in E^{bN}:\overline{\phi_a(\La_a)}^\Zar=X\}$ such that $P=\cup_{a\in M}\phi_a(\La_a)$.

In the second step, we construct a correspondence $\Gamma$ from $Z:=\bigsqcup_{a\in M}Z_a$ to $Z$. Precisely, for every $(a_1,a_2)\in M^2$, we set
$$\La_{a_1,a_2}:=\{(\xi_1,\xi_2)\in\La_{a_1}\times\La_{a_2}:f(\phi_{a_1}(\xi_1))=\phi_{a_2}(\xi_2)\}$$
and $\Gamma_{a_1,a_2}:=\overline{\La_{a_1,a_2}}^\Zar$. We define $\Gamma:=\bigsqcup_{(a_1,a_2)\in M^2}\Gamma_{a_1,a_2}$. Then $\Gamma\subseteq Z\times Z$ and it can be viewed as a correspondence from $Z$ to $Z$. Let $\phi=\bigsqcup_{a\in M}\phi_a:Z\to X$. We prove that $\pi_1(\Gamma)=Z$, where $\pi_1:Z\times Z\to Z$ is the first projection, and that $\overline{\phi\times\phi(\Gamma)}^\Zar=\Gamma_f$, the graph of $f$ in $X\times X$.

In the third step, from the correspondence $\Gamma$, we construct a correspondence $\psi$ from $Y$ to $Y$ and a morphism $\phi: Y\to X$ after several reductions, satisfying the following properties:
\begin{enumerate}
	\item $Y=\G_m^{\gamma}$ for some integer $\gamma\geq d=\dim(X)$, and $\overline{\phi(Y)}^\Zar=X$;
	\item $\psi$ is a torsion coset (hence irreducible) in $Y\times Y=\G_m^{2\gamma}$ with $\pi_1(\psi)=Y$;
	\item $\overline{\phi\times\phi(\psi)}^\Zar=\Gamma_f$;
	\item \label{skt4} $\psi(Y):=\pi_2(\psi)=Y$.
\end{enumerate}
The well-known torsion points theorem (Theorem \ref{thmtorsion}) implies that every $Z_a$ (and $\Gamma_{a_1,a_2}$) is a finite union of torsion cosets. After replacing $f$ by a suitable iterate, we can choose $Y$ to be a suitable component of some $Z_a$ and $\psi$ to be a suitable component of $\Gamma\cap(Z_a\times Z_a)$. To obtain \eqref{skt4}, we replace $Y$ by its image under some iterate of $\psi$.

In the final step, we show that $\psi$ can be made into the graph of an algebraic group endomorphism $g:Y\to Y$ after further reductions. We replace $X$ by its normalization in $\overline{K}(Y)$. Set $T:=\cap_{y\in Y}T_y$, where $T_y:=\Stab_Y(F_y)$ is the stabilizer of the fiber $F_y:=\phi^{-1}(\phi(y))$ in $Y$. Then $\phi:Y\to X$ factors through the quotient $Y\to Y/T$. After replacing $Y$ by $Y/T$, we may assume that $T=1$ is trivial. Using the algebraic group structure of $Y$, we prove that in this case $\psi$ is the graph of a morphism $g:Y\to Y$. Since $\psi$ is a torsion coset, after a further iterate and changing the group structure on $Y$, we can assume that $g$ is a (surjective) algebraic group endomorphism, which completes the proof.

\medskip

As for the proof of Theorem \ref{chdimequal}, a result of cohomological hyperbolicity (Lemma \ref{chdeggrow}) and the theory of linear tori (see \S \ref{sectori}) are the main ingredients.

\section{Preliminaries}
We recall the theorem of Loxton on cyclotomic integers in \S \ref{secLoxton}. In \S \ref{sectori}, we recall basic notions and results of linear tori. In \S \ref{secdyndeg}, we recall some useful properties of dynamical degrees. 

\subsection{Cyclotomic extensions and Loxton theorem}\label{secLoxton}
Let $k$ be a field of characteristic zero. We denote the group of all roots of unity in $k$ by $U(k)$.

Fix an algebraic closure $\overline{k}$ of $k$. The \emph{maximal cyclotomic extension} of $k$ (in $\overline{k}$) is the field $k^c:=k(U(\overline{k}))$, i.e., the subfield of $\overline{k}$ generated over $k$ by all roots of unity in $\overline{k}$.

Suppose now that $k$ is a number field. A theorem of Loxton \cite[Theorem 1]{Loxton} implies that there exists a suitable (non-decreasing) function $L:\R_{\geq0}\to\R_{\geq0}$ such that for every algebraic integer $\alpha\in\Q^c$, $\alpha$ can be written as a sum of (not necessarily distinct) roots of unity, $\alpha=\sum_{i=1}^b\xi_i$ with $0\leq b\leq L(C(\alpha))$. Such a function $L$ is called a \emph{Loxton function}. One can take $L(x)\ll_{\varepsilon}x^{2+\varepsilon}$ as $x\to+\infty$, for an arbitrary $\varepsilon>0$. Dvornicich and Zannier \cite[Theorem L]{DZ07} extended Loxton's theorem to cyclotomic extensions of an arbitrary number field $k$:
\begin{Thm}[Dvornicich-Zannier]\label{thmL}
	Let $L$ be any Loxton function and $k$ be a number field. There exist a constant $B=B_k\in\R_{>0}$ and a finite set $E=E_k\subset k$ with $\#E\leq[k:\Q]$ such that every algebraic integer $\alpha\in \sO_{k^c}$ can be written as $\alpha=\sum_{i=1}^b\eta_i\xi_i$, where $\eta_i\in E$ and $\xi_i\in U(\overline{k})$ for $1\leq i\leq b$, with $0\leq b\leq(\#E) L(B\cdot C(\alpha))$.
\end{Thm}
We deduce the following easy consequence of Theorem \ref{thmL}:
\begin{Cor}\label{thmLv}
	Let $k$ be a number field. Fix an integer $M\geq1$ and a constant $c\in\R_{>0}$. Then there exist $b=b(k,M,c)\in\Z_{>0}$ and a finite set $E=E(k,M,c)\subset k$ containing $0$ with $\#E\leq[k:\Q]+1$ such
	that for every $\alpha\in \frac{1}{M}\sO_{k^c}$ with $C(\alpha)\leq c$, we can write $\alpha$ as a sum $\alpha=\sum_{i=1}^b\eta_i\xi_i$, where $\eta_i\in E$ and $\xi_i\in U(\overline{k})$ for $1\leq i\leq b$.
\end{Cor}
\begin{proof}[Proof of Corollary \ref{thmLv} from Theorem \ref{thmL}]
	Fix a (non-decreasing) Loxton function $L$. Let $B=B_k\in\R_{>0}$ and $E_0=E_k\subset k$ be given by Theorem \ref{thmL}.
	
	Set $b=\lceil (\#E_0)L(BMc)\rceil$ and $E=\{0\}\cup\{\eta/M:\eta\in E_0\}$. For every $\alpha\in \frac{1}{M}\sO_{k^c}$ with $C(\alpha)\leq c$, we have $M\alpha\in \sO_{k^c}$ and
	$C(M\alpha)=M\cdot C(\alpha)\leq Mc$,
	so the desired property follows from Theorem \ref{thmL}.
\end{proof}

\subsection{Linear tori and torsion points theorem}\label{sectori}
We work over an algebraically closed field $k$ of characteristic zero (for example, $k=\overline{\Q}$ or $\C$) in this subsection. We recall some basic notions and results about the algebraic group $\G_m^n$. See \cite[Chapter 3]{HtBook} for a detailed treatment.

Let $n\geq1$ be an integer. The \emph{linear torus of dimension $n$} is the commutative algebraic group $\G_m^n$, where $\G_m=\Spec (k[t,t^{-1}])$ is the $1$-dimensional multiplicative group (over $k$). The multiplication on $\G_m^n$ is given by pointwise multiplication. Every algebraic group endomorphism $\varphi:\G_m^n\to\G_m^n$ is of the form
$$\varphi=\varphi_A:(u_1,\dots,u_n)\mapsto(u_1^{a_{11}}\cdots u_n^{a_{1n}},\dots,u_1^{a_{n1}}\cdots u_n^{a_{nn}}),$$
for some integral matrix $A=(a_{ij})\in M_n(\Z)$ (see \cite[Proposition 3.2.17]{HtBook}). As in \cite{PR04}, $A$ (or $\varphi_A$) is called \emph{positive} if every eigenvalue of $A$ is neither $0$ nor a root of unity. Clearly, we have $\varphi_{AB}=\varphi_A\circ\varphi_B$ for $A,B\in M_n(\Z)$.

Every (not necessarily irreducible) algebraic subgroup $H$ of $\G_m^n$ is of the form
$$H=H_{\La}=\{u=(u_1,\dots,u_n)\in\G_m^n:u^a:=u_1^{a_1}\cdots u_n^{a_n}=1,\forall a=(a_1,\dots,a_n)\in \La\},$$
where $\La$ is a subgroup of $\Z^n$ (see \cite[Theorem 3.2.19(a)]{HtBook}). The subgroup $H_\La$ is irreducible if and only if $\La$ is \emph{primitive}, i.e., $(\La\otimes_\Z\R)\cap\Z^n=\La$ (see \cite[Corollary 3.2.8]{HtBook}).

A \emph{linear subtorus} $H$ of $\G_m^n$ is an irreducible algebraic subgroup of $\G_m^n$, which must be isomorphic to $\G_m^{\dim(H)}$ \cite[Corollary 3.2.8]{HtBook}.

A \emph{torsion coset} in $\G_m^n$ is an irreducible subvariety of $\G_m^n$ of the form $\varepsilon\cdot H$, where $H$ is a linear subtorus of $\G_m^n$ and $\varepsilon\in \G_m^n(k)_\tors$ is a torsion point in $\G_m^n$. Note that the set $X\cap\G_m^n(k)_\tors$ of torsion points in a torsion coset $X$ is Zariski-dense in $X$ (see \cite[Proposition 3.3.6]{HtBook}).

We recall the famous torsion points theorem (see \cite[Theorem 1.1]{UnlikelyBook}) regarding torsion points in $\G_m^n$, which solves the higher-dimensional version of an old problem of Lang \cite[p.~201]{LangDio}. It was proved by Laurent \cite{Laurent}. See also Sarnak and Adams \cite{SA94} for a different proof.
\begin{Thm}[Torsion points theorem]\label{thmtorsion}
	Let $\Sigma\subseteq\G_m^n(k)_\tors$ be a subset of torsion points of $\G_m^n$. Then the Zariski closure of $\Sigma$ is a finite union of torsion cosets in $\G_m^n$. In particular, for an irreducible subvariety $X$ of $\G_m^n$, $X\cap\G_m^n(k)_\tors$ is Zariski-dense in $X$ if and only if $X$ is a torsion coset.
\end{Thm}
\begin{Rem}
In Theorem \ref{thmtorsion}, if we replace $\G_m^n$ by an arbitrary abelian variety and torsion cosets by cosets of abelian subvarieties by torsion points, then we obtain the classical Manin-Mumford conjecture, proved by Raynaud \cite{RaynaudMM}. Hence, Theorem \ref{thmtorsion} is also called the ``multiplicative Manin-Mumford''.
\end{Rem}
The following lemma follows from \cite[Proposition 6.1]{PR04} directly.
\begin{Lem}\label{posperdense}
	Let $n\geq1$ and $A\in M_n(\Z)$ be positive. Then the set of $\varphi_A$-periodic points in $\G_m^n(k)$ is Zariski-dense in $\G_m^n$.
\end{Lem}
For $A\in M_n(\Z)$ with $\det(A)\neq0$, we can always decompose $\varphi_A$ into the product of a factor $\varphi_{A_1}$ with all eigenvalues roots of unity, and a positive factor $\varphi_{A_2}$, up to a ramified cover.
\begin{Lem}\label{Gmndecomp}
	Let $n\geq1$ and $A\in M_n(\Z)$ with $\det(A)\neq0$. Then there exist integers $n_1,n_2\geq0$ with $n_1+n_2=n$, $A_1\in M_{n_1}(\Z)$ with all eigenvalues in $U(\overline{\Q})$, a positive $A_2\in M_{n_2}(\Z)$, and $P\in M_n(\Z)$ with $\det(P)\neq0$ such that $\varphi_A\circ\varphi_P=\varphi_P\circ g$, where $g=(\varphi_{A_1},\varphi_{A_2}):\G_m^n=\G_m^{n_1}\times\G_m^{n_2}\to\G_m^n=\G_m^{n_1}\times\G_m^{n_2}$ is given by $g(u_1,u_2)=(\varphi_{A_1}(u_1),\varphi_{A_2}(u_2))$.
\end{Lem}
\begin{proof}
By linear algebra, take an invertible $Q\in\GL_n(\Q)$ such that $Q^{-1}AQ=\diag(A_1,A_2)$, where $n_1,n_2\geq0$ are integers with $n=n_1+n_2$, $A_1\in M_{n_1}(\Z)$ is a matrix with all eigenvalues in $U(\overline{\Q})$, and $A_2\in M_{n_2}(\Z)$ is positive. Fix an integer $N\geq1$ such that $NQ\in M_n(\Z)$. Then $P:=NQ$ satisfies $\varphi_A\circ\varphi_P=\varphi_P\circ(\varphi_{A_1},\varphi_{A_2})$.
\end{proof}

\subsection{Dynamical degrees}\label{secdyndeg}
Let $X$ be a quasi-projective variety of dimension $d\geq1$ over a field $k$ of characteristic zero, and let $f:X\dashrightarrow X$ be a dominant rational self-map of $X$. Denote the locus of indeterminacy of $f$ by $I(f)$. For a subset $U$ of $X$, let $U_f$ be the set of points in $U$ whose forward $f$-orbit is well-defined.

We have already defined the dynamical degrees $\la_i(f)$ ($0\leq i\leq d$) and the cohomological Lyapunov multipliers $\mu_i(f)$ ($1\leq i\leq d+1$) of $f$. We always have $\la_0(f)=1$. For every $l\in\Z_{>0}$, we have $\la_i(f^{\circ l})=\la_i(f)^l$ for $0\leq i\leq d$.

Let $Y$ be a quasi-projective variety over $k$ and $g:Y\dashrightarrow Y$ a dominant rational self-map of $Y$. Assume that there exists a dominant rational map $h:Y\dashrightarrow X$ such that $f\circ h=h\circ g$. Then we have $\la_{i}(g)\geq\la_{i}(f)$ for all $0\leq i\leq d$, and equalities hold if $h$ is generically finite. This can be proved by the theory of relative dynamical degrees; see \cite{DN11,Dang2020,Truong2020}. In particular, dynamical degrees and cohomological Lyapunov multipliers are birational invariants.

For $1\leq i\leq d$, $f$ is called \emph{$i$-cohomologically hyperbolic} if $\mu_i(f)>1>\mu_{i+1}(f)$, or equivalently, if $i$ is the unique integer in $\{0,1,\dots,d\}$ such that $\la_i(f)=\max\{\la_j(f):0\leq j\leq d\}$. Then $f$ is cohomologically hyperbolic if and only if $f$ is $i$-cohomologically hyperbolic for some $1\leq i\leq d$.

Assume now that $X$ is projective. Let $L$ be a line bundle on $X$ and $n\geq1$ be an integer. Let $C\subseteq X$ be a curve such that $C\nsubseteq I(f^{\circ n})$. Let $\Gamma$ be the graph of $f^{\circ n}$ in $X\times X$ and $\pi_j:\Gamma\to X$ be the $j$-th projection for $j=1,2$. Then $C\nsubseteq I(\pi_1^{-1})$. Define $(L_n\cdot C):=(\pi_2^*L\cdot C_\pi)$, where $C_\pi:=\overline{\pi_1^{-1}(C\setminus I(\pi_1^{-1}))}^\Zar$ is the strict transform of $C$ in $\Gamma$. The intersection number $(L_n\cdot C)$ is the desired one for ``$\left((f^{\circ n})^*(L)\cdot C)\right)$'', which can be computed using any sufficiently high projective models of $X$. See \cite{recursive} for more information.

The following lemma is \cite[Lemma 6.5]{recursive}:
\begin{Lem}\label{chdeggrow}
	Let $X$ be a projective variety of dimension $d\geq1$ over a field $k$ of characteristic zero, and let $f:X\dashrightarrow X$ be a dominant rational self-map of $X$. Let $L$ be an ample line bundle on $X$. Assume that $f$ is $i$-cohomologically hyperbolic where $1\leq i\leq d$. Then for every $1\leq \beta<\mu_i(f)$, there exists a non-empty affine Zariski open subset $U\subseteq X$ such that for every irreducible curve $C\subseteq X$ with $C\cap U_f\neq\emptyset$ and $\dim(f^{\circ n}(C))=1$ for all $n\geq1$, we have
	$$\liminf_{n\to\infty}(L_n\cdot C)^{1/n}\geq\beta.$$
\end{Lem}
The dynamical degrees of group endomorphisms of $\G_m^n$ are known, see \cite[Theorem 1]{Lin12} or \cite[Corollary B]{FW12}.
\begin{Prop}\label{Gmndyndeg}
	Let $n\geq1$ and $A\in M_n(\Z)$ with $\det(A)\neq0$. Consider the group endomorphism $\varphi_A:\G_m^n\to\G_m^n$ over a field $k$ of characteristic zero. Let $\nu_1,\dots,\nu_n$ be the eigenvalues of $A$ in $\C$ (counted with multiplicity) such that $\left|\nu_1\right|\geq\cdots\geq\left|\nu_n\right|>0$. Then $\la_i(\varphi_A)=\left|\nu_1\cdots\nu_i\right|$ for $1\leq i\leq n$. In particular, $\varphi_A$ is cohomologically hyperbolic if and only if $A$ has no eigenvalues with absolute value $1$.
\end{Prop}

\section{Proof of the main results}\label{seccycmainpf}
\begin{proof}[Proof of Theorem \ref{thmmain}]
We say a subset $Q\subset X(\overline{K})$ is exceptional if $\overline{Q}^\Zar\subsetneq X$. Note that $P\setminus Q$ still satisfies the conditions (DCI), (BH), and (AI) for every exceptional $Q\subset X(\overline{K})$, since the endomorphism $f:X\to X$ is dominant. Thus, we are free to remove an exceptional subset from $P$.

The following proof is divided into several steps.

\medskip

{\bfseries Step 1: Apply Loxton's Theorem.}

We fix an integer $b=b(K,M,c)\in\Z_{>0}$ and a subset $0\in E=E(K,M,c)\subset K$ with $\#E\leq[K:\Q]+1$ as in Corollary \ref{thmLv}, such that for every $\alpha\in\frac{1}{M}\sO_{K^c}$ with $C(\alpha)\leq c$, we have $\alpha\in\sum_{j=1}^b E\cdot U(\C)$. By the conditions (DCI) and (BH), we see that for every $y=(y_1,\dots,y_N)\in P$, we have $y_i\in\sum_{j=1}^b E\cdot U(\C)$, $1\leq i\leq N$.

Set $M_1=E^{bN}$ and $G=\G_{m,K}^{bN}$. Note that $M_1$ is a finite set. We say a point $y=(y_1,\dots,y_N)\in X(\overline{K})$ has a lift $\xi=(\xi_{ij})_{1\leq i\leq N,1\leq j\leq b}\in G(\overline{K})_\tors$ of type $a=(a_{ij})_{1\leq i\leq N,1\leq j\leq b}\in M_1$ if $y_i=\sum_{j=1}^b a_{ij}\xi_{ij}$ for all $1\leq i\leq N$.

For every $a=(a_{ij})\in M_1$, we set
\begin{align*}
	\La_a&=\{\xi\in G(\overline{K})_\tors:\xi\text{ is a lift of some point }y\in P\text{ of type }a\},\\
	\phi_a&:G\to\A^N,(z_{ij})_{1\leq i\leq N,1\leq j\leq b}\mapsto\left(\sum_{j=1}^b a_{1j}z_{1j},\dots,\sum_{j=1}^b a_{Nj}z_{Nj}\right)
\end{align*}
Then for every $a\in M_1$, the following statements hold:
\begin{itemize}
	\item $\xi\in G(\overline{K})_\tors$ is a lift of $y\in P$ if and only if $\phi_a(\xi)=y$;
	\item $\phi_a(\La_a)\subseteq P\subseteq X(K^c)$.
\end{itemize}

Set $M=\{a\in M_1:\La_a\neq\emptyset\}\subseteq M_1=E^{bN}$. After replacing $P$ by $P\setminus Q$ for a suitable exceptional subset $Q\subset P$, we may assume that
\begin{equation}\label{phiZadense}
	\overline{\phi_a(\La_a)}^\Zar=X
\end{equation} for every $a\in M$. Note that we have $\cup_{a\in M}\phi_a(\La_a)=P$.

From now on, we make base change to $\overline{K}$ for all varieties and morphisms, and we omit the notation $\overline{K}$ in subscript for simplicity.

For every $a\in M$, let $G_a:=G$ be a copy of $G$ indexed by $a$, and let $Z_a:=\overline{\La_a}^\Zar\subseteq G_a$. We still denote $\phi_a:G_a=G\to\A^N$. Clearly, $\phi_a(Z_a)\subseteq X$ since $\overline{\phi_a(\La_a)}^\Zar=X$. For every exceptional $Q\subset P$, we set $Z_{a,Q}:=\overline{\La_a\setminus\phi_a^{-1}(Q)}^\Zar\subseteq G_a$. By noetherianity, after replacing $P$ by $P\setminus P_\excp$ for a suitable exceptional subset $P_\excp\subset P$, we may assume that
\begin{equation}\label{Zexp}
Z_a=Z_{a,Q}
\end{equation}
for every $a\in M$ and exceptional $Q\subset P$.

Then for every $a\in M$ and every irreducible component $Z_0$ of $Z_a$, we have
\begin{equation}\label{compnotpt}
	\dim(Z_0)\geq1.
\end{equation}
(Otherwise, if $Z_0=\{\zeta\}$ is an irreducible component of $Z_a$ of dimension $0$ for some $a\in M$, then $\zeta\notin Z_{a,\{\phi_a(\zeta)\}}=Z_a$, a contradiction.)

\medskip

{\bfseries Step 2: Construct a Correspondence.}

Using the torsion coset theorem, we construct a correspondence $\Gamma$ from $Z$ to $Z$ in this step.

For all $a_1,a_2\in M$, set
$$\La_{a_1,a_2}:=\{(\xi_1,\xi_2)\in\La_{a_1}\times\La_{a_2}:f(\phi_{a_1}(\xi_1))=\phi_{a_2}(\xi_2)\}\subseteq G_{a_1}(\overline{K})_\tors\times G_{a_2}(\overline{K})_\tors$$
and $\Gamma_{a_1,a_2}:=\overline{\La_{a_1,a_2}}^\Zar\subseteq G_{a_1}\times G_{a_2}$.

Set
\begin{align*}
	\La&:=\bigsqcup_{(a_1,a_2)\in M^2}\La_{a_1,a_2}\subseteq \bigsqcup_{(a_1,a_2)\in M^2} G_{a_1}\times G_{a_2}=\left(\bigsqcup_{a\in M}G_a\right)^2,\\
	\Gamma&:=\overline{\La}^\Zar=\bigsqcup_{(a_1,a_2)\in M^2}\Gamma_{a_1,a_2}\subseteq \left(\bigsqcup_{a\in M}G_a\right)^2,\text{ and}\\
	Z&:=\bigsqcup_{a\in M}Z_a\subseteq \bigsqcup_{a\in M}G_a.
\end{align*}

From the construction, for every $(a_1,a_2)\in M^2$, torsion points are dense in the closed subset $\Gamma_{a_1,a_2}\subseteq G_{a_1}\times G_{a_2}=\G_m^{2bN}$, so $\Gamma_{a_1,a_2}$ is a finite union of torsion cosets in $G_{a_1}\times G_{a_2}$ by Theorem \ref{thmtorsion}. Similarly, for every $a\in M$, the closed subset $Z_a\subseteq G_a$ is a finite union of torsion cosets in $G_a$.

Denote the first and second projections $(\bigsqcup_{a\in M}G_a)^2\to\bigsqcup_{a\in M}G_a$ by $\pi_1$ and $\pi_2$, respectively. 

We show that $\pi_1(\Gamma)=Z$. Since $\Gamma_{a_1,a_2}$ is a finite union of torsion cosets in $G_{a_1}\times G_{a_2}$ for all $a_1,a_2\in M$, by \cite[Proposition 3.2.18]{HtBook}, the set $\pi_1(\Gamma)$ is closed in $\bigsqcup_{a\in M}G_a$. Clearly, we have $\pi_1(\Gamma)\subseteq Z$ from the definition. Fix an arbitrary $a\in M$. For every $\xi\in\La_a\setminus\phi_a^{-1}(P\setminus f^{-1}(P))$, since $\cup_{a_2\in M}\phi_{a_2}(\La_{a_2})=P$, there exists $(\xi,\xi_2)\in\Gamma_{a,a_2}$ for some $a_2\in M$ and $\xi_2\in\Gamma_2$. Thus
$$\La_a\setminus\phi_a^{-1}(P\setminus f^{-1}(P))\subseteq\bigcup_{a_2\in M}\pi_1(\Gamma_{a,a_2})=\pi_1(\Gamma)\cap G_a.$$
Taking closure, we get
$$Z_a=Z_{a,P\setminus f^{-1}(P)}=\overline{\La_a\setminus\phi_a^{-1}(P\setminus f^{-1}(P))}^\Zar\subseteq \pi_1(\Gamma)\cap G_a,$$
where the first equality follows from \eqref{Zexp} and (AI). As $a\in M$ is arbitrary, we deduce that $Z=\bigsqcup_{a\in M}Z_a\subseteq\pi_1(\Gamma)$, hence $\pi_1(\Gamma)=Z$.

From the definition, we see that $\pi_2(\Gamma)\subseteq Z$, so $\Gamma\subseteq Z\times Z$. We view the closed subset $\Gamma$ (with the reduced structure) as a correspondence from $Z$ to $Z$.

Define a morphism $\phi=\bigsqcup_{a\in M}\phi_a:Z\to X$ by $\phi=\phi_a$ on $Z_a$ for every $a\in M$.

Denote the graph of $f$ by $\Gamma_f\subseteq X\times X$, which is closed in $X\times X$ and can be viewed as a correspondence from $X$ to $X$. Let $a\in M$ be an arbitrary element in $M$ and set $\Gamma_a:=\bigsqcup_{a_2\in M}\Gamma_{a,a_2}$. Next we show
$$\overline{\phi\times\phi(\Gamma)}^\Zar=\bigcup_{a\in M}\overline{\phi\times\phi(\Gamma_a)}^\Zar=\Gamma_f\subseteq X\times X.$$
Indeed, from the definition we have $\phi\times\phi(\La)\subseteq \Gamma_f$, then $\phi\times\phi(\Gamma)=\phi\times\phi(\overline{\La}^\Zar)\subseteq \Gamma_f$, and hence
$$\overline{\phi\times\phi(\Gamma_a)}^\Zar\subseteq\overline{\phi\times\phi(\Gamma)}^\Zar\subseteq\Gamma_f.$$
Note that $\phi(Z_a)=\phi_a(Z_a)$ is Zariski-dense in $X$ by assumption \eqref{phiZadense}. We still denote the first projection $X\times X\to X$ by $\pi_1$. Then
$$\pi_1(\phi\times\phi(\Gamma_a))=\phi(\pi_1(\Gamma_a))=\phi(Z_a)$$
is Zariski-dense (and contained) in $X$, since $\pi_1(\Gamma_a)=\pi_1(\Gamma)\cap G_a=Z\cap G_a=Z_a$. Thus
$$\dim(\overline{\phi\times\phi(\Gamma_a)}^\Zar)\geq\dim(\overline{\pi_1(\phi\times\phi(\Gamma_a))}^\Zar)=\dim(X)=d=\dim(\Gamma_f),$$
so we must have $\overline{\phi\times\phi(\Gamma)}^\Zar=\Gamma_f$, since $\Gamma_f\cong X$ is geometrically irreducible and $\overline{\phi\times\phi(\Gamma_a)}^\Zar\subseteq\Gamma_f$.

\medskip

{\bfseries Step 3: Irreducibility and Surjectivity.}

In this step, we make further reductions to show that from the correspondence $\Gamma$ from $Z$ to $Z$, we can construct a correspondence $\psi$ from $Y$ to $Y$ such that both $Y$ and $\psi\subseteq Y\times Y$ are irreducible, and $\psi$ is surjective (i.e., $\pi_2(\psi)=Y$).

Recall that each $Z_a$ is a finite union of torsion cosets in $G_a=\G_m^{bN}$. Then we can write $Z=\bigsqcup_{a\in M}Z_a$ as a finite union $Z=\cup_{\alpha\in I}Y_\alpha$, where $I$ is a non-empty index set and for every $\alpha\in I$, we have $Y_\alpha\cong\G_m^{\gamma_\alpha}$ as $\overline{K}$-varieties for some $\gamma_\alpha\in\Z_{>0}$. (We can take the $Y_\alpha$'s to be all the irreducible components of $Z$. We have shown that the dimension $\gamma_\alpha$ is strictly positive; see \eqref{compnotpt}.)

Replacing $Z$ by $\bigsqcup_{\alpha\in I}Y_\alpha$ and lifting $\La\subseteq Z\times Z,\Gamma\subseteq Z\times Z$ (resp. $\phi:Z\to X$) to $(\bigsqcup_{\alpha\in I}Y_\alpha)\times (\bigsqcup_{\alpha\in I}Y_\alpha)$ (resp. $\bigsqcup_{\alpha\in I}Y_\alpha$) via the natural morphism $\bigsqcup_{\alpha\in I}Y_\alpha\to Z$, we may assume that $Z=\bigsqcup_{\alpha\in I}Y_\alpha$ is a disjoint union.

For each $\alpha\in I$, we identify $Y_\alpha=\G_m^{\gamma_\alpha}$. Precisely, for a given $\alpha\in I$, assume that $Y_\alpha=\varepsilon_\alpha\cdot H_\alpha\subseteq G_a=\G_m^{bN}$ for some $a=a(\alpha)\in M$, a torsion point $\varepsilon_\alpha\in G_a(\overline{K})_\tors$, and a linear subtorus $H_\alpha\leqslant G_a$ of dimension $\gamma_\alpha$. We identify $Y_\alpha$ with $\G_m^{\gamma_\alpha}$ via the isomorphism $\G_m^{\gamma_\alpha}\cong H_\alpha\xrightarrow{\sim} \varepsilon_\alpha\cdot H_\alpha,H_\alpha\ni y\mapsto\varepsilon_\alpha y\in\varepsilon_\alpha\cdot H_\alpha$ of $\overline{K}$-varieties.

We write $\Gamma=\bigsqcup_{(\alpha,\beta)\in I^2}\Gamma_{\alpha,\beta}$, where $\Gamma_{\alpha,\beta}=\Gamma\cap (Y_\alpha\times Y_\beta)$. For every $(\alpha,\beta)\in I^2$, note that the closed set $\Gamma_{\alpha,\beta}$ is also a finite union of torsion cosets in $\Gamma_{\alpha,\beta}=\G_m^{\gamma_\alpha+\gamma_\beta}$, since the points $\varepsilon_\alpha,\varepsilon_\beta$ appearing in the identification $Y_\alpha=\G_m^{\gamma_\alpha},Y_\beta=\G_m^{\gamma_\beta}$ are torsion.

Observe that we still have $\overline{\phi(Z)}^\Zar=X$, $\pi_1(\Gamma)=Z$, and $\overline{\phi\times\phi(\Gamma)}^\Zar=\Gamma_f$ in $ X\times X$.

For each $\alpha\in I$, define $I_\alpha:=\{\beta\in I:\pi_1(\Gamma_{\alpha,\beta})=Y_\alpha)\}$, which is non-empty. Indeed, for all $\beta\in I$, the image $\pi_1(\Gamma_{\alpha,\beta})$ is closed in $Y_\alpha$ by \cite[Proposition 3.2.18]{HtBook}, as $\Gamma_{\alpha,\beta}$ is a finite union of torsion cosets. The equality $\pi_1(\Gamma)=Z=\bigsqcup_{\alpha\in I}Y_\alpha$ implies that $\cup_{\beta\in I}\pi_1(\Gamma_{\alpha,\beta})=Z\cup Y_\alpha=Y_\alpha$, so $I_\alpha\neq\emptyset$.

Set $J:=\{\alpha\in I:\overline{\phi(Y_\alpha)}^\Zar=X\}$, which is non-empty because $\overline{\phi(Z)}^\Zar=X$ and the index set $I$ is finite. It is clear that $\gamma_\alpha=\dim(Y_\alpha)\geq\dim(X)=d$ for every $\alpha\in J$.

Note that for every $\alpha\in J$, we have $I_\alpha\subseteq J$. Indeed, given an arbitrary $\beta\in I_\alpha$, the equalities $\overline{\phi(Y_\alpha)}^\Zar=X$ and $\pi_1(\Gamma_{\alpha,\beta})=Y_\alpha$ imply that $\overline{\phi\times\phi(\Gamma_{\alpha,\beta})}^\Zar=\Gamma_f$ as the argument in the last lines of Step 2 applies. Then we see that $\beta\in J$, since the endomorphism $X$ is dominant, i.e., $\overline{\pi_2(\Gamma_f)}^\Zar=X$. Therefore, we can pick and fix a map $\sigma:J\to J$ such that $\sigma(\alpha)\in I_\alpha$ for every $\alpha\in J$. 

Set $Y:=\bigsqcup_{\alpha\in J}Y_\alpha$ and $\psi:=\bigsqcup_{\alpha\in J}\Gamma_{\alpha,\sigma(\alpha)}:=\bigsqcup_{\alpha\in J}\psi_\alpha\subseteq Y\times Y$. We also view $\psi$ as a correspondence from $Y$ to $Y$, and $\psi_\alpha$ as a correspondence from $Y_\alpha$ to $Y_{\sigma(\alpha)}$ for $\alpha\in J$. We use $\phi$ to denote the restriction of $\phi$ on $Y$, as well. The following properties hold:
\begin{enumerate}
	\item \label{cycprop1} $\forall \alpha\in J$, $Y_\alpha=\G_m^{\gamma_\alpha}$ for some integer $\gamma_\alpha\geq d$, and $\overline{\phi(Y_\alpha)}^\Zar=X$;
	\item \label{cycprop2} $\forall \alpha\in J$, $\psi_\alpha$ is a finite union of torsion cosets in $Y_\alpha\times Y_{\sigma(\alpha)}=\G_m^{\gamma_\alpha+\gamma_{\sigma(\alpha)}}$ such that $\pi_1(\psi_\alpha)=Y_\alpha$;
	\item \label{cycprop3} $\forall \alpha\in J$, $\overline{\phi\times\phi(\psi_\alpha)}^\Zar=\Gamma_f$.
\end{enumerate}

Since the set $J$ is finite, we can pick and fix a $\sigma$-periodic point $\alpha_0\in J$. Assume that $n\in\Z_{>0}$ is the exact period of $\alpha_0$, so $\sigma^{\circ n}(\alpha_0)=\alpha_0$. Replacing $(f, Y)$ by $(f^{\circ n},Y_{\alpha_0})$ and $\psi$ by the composition $\psi_{\sigma^{\circ (n-1)}(\alpha_0)}\circ\cdots\circ\psi_{\alpha_0}$ of correspondences, we may assume that $Y=Y_{\alpha_0}$ is irreducible. (Here, $\psi_{\sigma^{\circ (n-1)}(\alpha_0)}\circ\cdots\circ\psi_{\alpha_0}$ as a reduced closed subscheme of $Y_{\alpha_0}\times Y_{\alpha_0}$ is defined by: a point $(\xi,\zeta)\in Y_{\alpha_0}\times Y_{\alpha_0}$ is in $\psi_{\sigma^{\circ (n-1)}(\alpha_0)}\circ\cdots\circ\psi_{\alpha_0}$ if and only if there exist points $\xi_i\in Y_{\sigma^{\circ i}(\alpha_0)}$ for $0\leq i\leq n$ with $\xi_0=\xi$ and $\xi_n=\zeta$ such that $(\xi_i,\xi_{i+1})\in\psi_{\sigma^{\circ i}(\alpha_0)}$ for $0\leq i\leq n-1$.) Note that the above properties \eqref{cycprop1}, \eqref{cycprop2}, and \eqref{cycprop3} still hold. (It is easy to see that the torsion points in the new $\psi$ are Zariski-dense, so \eqref{cycprop2} follows from the torsion points theorem.) From now on, the index set $J=\{\alpha_0\}$ is always assumed to be a singleton.

Replacing $\psi$ by a suitable irreducible component $\psi_0$ of $\psi\subseteq Y\times Y$, we may assume that $\psi\subseteq Y\times Y$ is irreducible and the properties \eqref{cycprop1}, \eqref{cycprop2}, and \eqref{cycprop3} still hold.

We set $\psi(Y)=\psi^{\circ 1}(Y)=\pi_2(\psi)\subseteq Y$ and define
$$\psi^{\circ k}(Y)=\pi_2(\pi_1^{-1}(\psi^{\circ (k-1)}(Y))\cap\psi)\subseteq Y$$
for $k\geq2$ inductively. Note that $(\psi^{\circ k}(Y))_{k\geq1}$ is a decreasing sequence of (irreducible) torsion cosets in $Y$ (the image of a torsion coset under an algebraic group morphism from $\G_m^{k_1}$ to $\G_m^{k_2}$ is still a torsion coset; see \cite[Proposition 3.2.18]{HtBook}). By noetherianity, we can take an integer $r\geq1$ such that $\psi^{\circ k}(Y)=\psi^{\circ r}(Y)$ for all integers $k\geq r$. Replacing $(f,Y,\psi)$ by $(f^{\circ r},\psi^{\circ r}(Y),\psi^{\circ r})$ and identifying $\psi^{\circ r}(Y)=\G_m^{\dim(\psi^{\circ r}(Y))}$, we may assume that $\psi(Y)=Y$. Note that the properties \eqref{cycprop1}, \eqref{cycprop2}, \eqref{cycprop3}, and \eqref{cycprop4} hold, where \eqref{cycprop4} is the following property:
\begin{enumerate}
	\setcounter{enumi}{3}
	\item \label{cycprop4} both $Y$ and $\psi$ are irreducible, and $\psi(Y)=Y$.
\end{enumerate}
\medskip

{\bfseries Step 4: Group Endomorphism.}

In this step, under further reductions, we induce a true group endomorphism $g:Y\to Y$ from $\psi$.

Replacing $X$ by its normalization in $\overline{K}(Y)$, we may assume that for a general $y\in Y(\overline{K})$, the fiber $F_y:=\phi^{-1}(\phi(y))\subseteq Y$ of $\phi(y)$ under $\phi$ is irreducible.

For each $y\in Y(\overline{K})$, set $T_y:=\Stab_Y(F_y)=\{t\in Y:t\cdot F_y=F_y\}$, which is an algebraic subgroup of $Y$. Here we use multiplicative notation for the group operation on $Y=\G_m^{\gamma_{\alpha_0}}$.

Set $T:=\cap_{y\in Y}T_y$. Then $T=T_{y_0}$ for any general $y_0\in Y(\overline{K})$.

Note that the morphism $\phi:Y\to X$ factors through the quotient $Y\to Y/T$ by the definition of $T$. Replacing $Y$ by $Y/T$ and $\psi$ by its image in $Y/T\times Y/T$, we may assume that $T=1$ is the trivial subgroup. The properties \eqref{cycprop1}, \eqref{cycprop2}, \eqref{cycprop3}, and \eqref{cycprop4} in Step 3 still hold:
\begin{enumerate}
	\item $Y=\G_m^{k}$ for some integer $k\geq d$, and $\overline{\phi(Y)}^\Zar=X$;
	\item $\psi$ is a finite union of torsion cosets in $Y\times Y=\G_m^{2k}$ such that $\pi_1(\psi)=Y$;
	\item $\overline{\phi\times\phi(\psi)}^\Zar=\Gamma_f$;
	\item $\psi$ is irreducible and $\psi(Y)=Y$.
\end{enumerate}

Pick any $y_1\in\psi(1)=\pi_2(\psi\cap\pi_1^{-1}(1))$ and set $V:=y_1^{-1}\cdot\psi(1)$. Since $\psi$ is a torsion coset in $Y\times Y$, $\psi(1)$ is a translation of an algebraic subgroup of $Y$. As $\psi(1)$ contains the point $y_1$, we see that $V$ is an algebraic subgroup of $Y$ (which may not be irreducible). Note that the algebraic group $V$ is independent of the choices of $1\in Y(\overline{K})$ and $y_1\in\psi(1)$; that is, for every $y\in Y(\overline{K})$ and $z\in\psi(y)$, we have $\psi(y)=z\cdot V$.

Pick a general $z\in Y(\overline{K})$. For every $y\in\psi^{-1}(F_z)$, we have
$$\phi(\psi(F_y))=\{f(\phi(y))\}=\{\phi(z)\},$$
so $\psi(F_y)\subseteq F_z$. Thus $\psi(\psi^{-1}(F_z))\subseteq F_z$. On the other hand, we have $F_z\subseteq\psi(\psi^{-1}(F_z))$ since $\pi_2(\psi)=Y$ by the property (4). We conclude that
$$F_z=\psi(\psi^{-1}(F_z))=\bigcup_{w\in\psi^{-1}(F_z)}\psi(w)=\bigcup_{w\in\psi^{-1}(F_z)}\tau(w)\cdot V,$$
which is $V$-invariant, where $\tau(w)$ is any chosen point in $\psi(w)$ for $w\in\psi^{-1}(F_z)$. Then $V\leqslant\Stab_{Y}(F_z)=T=1$, and hence $V=1$.

From the fact that $V=1$, we conclude that for every $y\in Y(\overline{K})$ and $z\in\psi(y)$, we have $\psi(y)=z\cdot V=\{z\}$, so $\psi\subseteq Y\times Y$ is the graph of some morphism $g:Y\to Y$ between schemes, since $\pi_1(\psi)=Y$.

Finally, we show that $g:Y\to Y$ can be made a group endomorphism. Since $\psi$ is a torsion coset, we can write $g(y)=\tau_0\cdot g_0(y)$ for $y\in Y$, where $g_0:Y\to Y$ is an algebraic group endomorphism and $\tau_0\in Y(\overline{K})_\tors$ is a torsion point. As $\psi$ is surjective, both $g$ and $g_0$ are surjective. Assume that $n\in\Z_{>0}$ is the order of $\tau_0$ in $Y(\overline{K})$. Then it is easy to see that the (forward) orbit $O_g(1)=\{1,g(1),g^{\circ 2}(1),\dots\}$ is contained in the finite set
$$\{\zeta\in Y(\overline{K})_\tors:\zeta^n=1\}\cong \{\xi\in\overline{K}:\xi^n=1\}^{\dim(Y)}.$$
Thus $1\in Y$ is $g$-preperiodic. After replacing $(f,g)$ by $(f^{\circ r},g^{\circ r})$ for a suitable $r\in\Z_{>0}$, we may assume that $g^{\circ 2}(1)=g(1)$. Then $y_0:=g(1)$ becomes a fixed torsion point for $g$. We can change the group structure on $Y$ such that the point $y_0$ becomes the identity element of $Y$, via the isomorphism $Y\xrightarrow{\cong}Y,y\mapsto y_0\cdot y$ of $\overline{K}$-varieties. Then we have $g(1)=1$ and $g$ is a surjective algebraic group endomorphism on $Y$. The morphism $\phi:Y\to X$ is dominant by property (1), and we have $f\circ\phi=\phi\circ g$. This completes the proof of the theorem.
\end{proof}
\begin{proof}[Proof of Theorem \ref{chdimequal}]
	It suffices to prove the ``only if'' direction.
	
	Assume that $(X,f)$ is of monomial type. After a base change, we may assume that $k=\overline{k}$ and work over $\overline{k}$. Then there exist integers $l\geq1,n\geq d$, a group endomorphism $g:\G_{m}^n\to\G_{m}^n$, and a dominant morphism $\phi:\G_{m}^n\to X$ such that $f^{\circ l}\circ\phi=\phi\circ g$. Without loss of generality, we may assume that $l=1$ (note that $f^{\circ l}$ remains cohomologically hyperbolic).
	
	By Lemma \ref{Gmndecomp}, we may assume that $g$ decomposes as $g=(\varphi_{A_1},\varphi_{A_2})$, where $n_1,n_2\geq0$ with $n_1+n_2=n$, $A_1\in M_{n_1}(\Z)$ has all eigenvalues in $U(k)$, and $A_2\in M_{n_2}(\Z)$ is positive. 
	
	Let $1\leq i\leq d$ be the unique integer such that $\mu_i(f)>1>\mu_{i+1}(f)$. Fix a projective compactification $X^\prime$ of $X$ and a very ample line bundle $L$ on $X^\prime$. View $f$ as a dominant rational self-map $f:X^\prime\dashrightarrow X^\prime$ of $X^\prime$. Embed $\G_m^q\subset\P^q$ for $q\geq0$. Fix $1<\beta<\mu_i(f)$ and let $U\subseteq X^\prime$ be the non-empty Zariski open subset given by Lemma \ref{chdeggrow}. Set $H:=X^\prime\setminus U$, which is a proper Zariski closed subset of $X^\prime$.
	
	Since $\phi$ is dominant, by Lemma \ref{posperdense}, for a general $\varphi_{A_2}$-periodic point $z\in\G_m^{n_2}(k)$, the image $\phi(\G_m^{n_1}\times\{z\})$ is not contained in $H$. For such a $z$, let $Y_z$ be the Zariski-closure of $\phi(\G_m^{n_1}\times\{z\})$ in $X^\prime$, which is irreducible.
	
	\textbf{Claim:} For a general $\varphi_{A_2}$-periodic point $z\in\G_m^{n_2}(k)$, $Y_z$ is a closed point in $X$.
	
	\textit{Proof of the Claim:} Suppose, for contradiction, that for a general $\varphi_{A_2}$-periodic point $z\in\G_m^{n_2}(k)$, we have $\dim(Y_z)\geq1$ and $Y_z\nsubseteq H$. Then we can find an irreducible curve $C_0$ in $\G_m^{n_1}\times\{z\}$ such that $C:=\overline{\phi(C_0)}^\Zar$ is an irreducible curve in $X^\prime$ with $C\cap U_f\neq\emptyset$ and $\dim(f^{\circ r}(C))=1$ for every $r\geq1$. After a suitable iterate, we may assume that $z$ is a $\varphi_{A_2}$-fixed point. Let $C_1$ be the Zariski-closure of $C_0$ in $\P^{n_1}\supseteq\G_m^{n_1}\times\{z\}$. For every $r\geq1$, pick a birational morphism $\tau_r:Y_r\to 
	\P^{n_1}$ from a projective variety $Y_r$ such that both $\phi_r:=\phi\mid_{\G_m^{n_1}\times\{z\}}\circ \tau_r:Y_r\to X^\prime$ and $g_r:=\phi\mid_{\G_m^{n_1}\times\{z\}}\circ\varphi_{A_1}^{\circ r}\circ \tau_r:Y_r\to X^\prime$ are morphisms, with $C_1\nsubseteq I(\tau_r^{-1})$. Define $(L^\prime_r\cdot C_0):=(g_r^*(L)\cdot C_2)$, where $C_2=\overline{\tau_r^{-1}(C_1\setminus I(\tau_r^{-1}))}^\Zar$. (Then the intersection number $(L^\prime_r\cdot C_0)$ is the desired one for ``$((\varphi_{A_1}^{\circ r})^*((\phi\mid_{\G_m^{n_1}\times\{z\}})^*L)\cdot C_0) $'', which can be computed using any sufficiently high projective models of $\G_m^{n_1}\times\{z\}$; see \cite{recursive} for details.) Let $\beta_1:=(\beta+2)/3$ and $\beta_2:=(2\beta+1)/3$, then $1<\beta_1<\beta_2<\beta$. By Lemma \ref{Gmndyndeg}, all dynamical degrees of $\varphi_{A_1}$ on $\G_{m}^{n_1}\cong\G_m^{n_1}\times\{z\}$ are $1$ since the eigenvalues of $A_1$ are roots of unity. Therefore, there exists a constant $A_1>0$ such that $(L^\prime_r\cdot C_0)\leq A_1\beta_1^r$ for all $r\geq1$. However, Lemma \ref{chdeggrow} implies that there exists a constant $A_2>0$ such that $(L_r\cdot C)\geq A_2\beta_2^r$ for all $r\geq1$. Then $(L_r\cdot C)>(L^\prime_r\cdot C_0)$ for all sufficiently large integers $r\gg1$, contradicting the projection formula. Therefore, the claim holds.
	
	Since the condition that $\overline{\phi(\G_m^{n_1}\times\{z\})}^\Zar$ is a closed point is a closed condition in $z\in \G_m^{n_2}$, and by Lemma \ref{posperdense} the set of $\varphi_{A_2}$-periodic points is Zariski-dense, we conclude that $\phi(\G_m^{n_1}\times\{z\})$ is a closed point in $X$, for every $z\in\G_m^{n_2}$. Thus, there exists a morphism $\phi^\prime:\G_m^{n_2}\to X$ such that $\phi=\phi^\prime\circ\pi_2$, where $\pi_2:\G_m^n=\G_m^{n_1}\times\G_m^{n_2}\to\G_m^{n_2}$ is the projection. Note that $\phi^\prime$ is dominant as $\phi$ is.
	
	Replacing $(\G_m^n,g,\phi)$ by $(\G_m^{n_2},\varphi_{A_2},\phi^\prime)$, we may assume that $n_1=0$ and $g=\varphi_{A_2}$ is positive.
	
	For $y\in \G_m^n(k)$, set $F_y:=\phi^{-1}(\phi(y))$. By the argument in Step 4 of the proof of Theorem \ref{thmmain}, we may assume that $F_y$ is irreducible for a general $y\in\G_m^n(k)$, and that $T=1$ is the trivial subgroup, where $T=\Stab_{\G_m^n}(F_y)$ for a general $y\in \G_m^n(k)$.
	
	We now show that $n=\dim(X)$. Let $Z$ be the irreducible component of $\G_m^n\times_X\G_m^n=\{(u,v)\in\G_m^n\times\G_m^n:\phi(u)=\phi(v)\}\subseteq\G_m^n\times\G_m^n$ containing the diagonal $\Delta$ in $\G_m^n\times\G_m^n$. Then $Z$ is $(g,g)$-invariant, where $(g,g):\G_m^n\times\G_m^n\to \G_m^n\times\G_m^n$ is the positive group endomorphism given by $(g,g)(u,v)=(g(u),g(v))$. By \cite[Proposition 6.1]{PR04}, $(g,g)$-periodic torsion points are Zariski-dense in $Z$. By the torsion points theorem (Theorem \ref{thmtorsion}), the irreducible closed subset $Z$ is a linear subtorus (containing $\Delta$) in $\G_m^n\times\G_m^n$ because the identity element is in $Z$. For a general $y\in \G_m^n(k)$,
	$$F_y=\phi^{-1}(\phi(y))=\{u\in\G_m^n:(1,u)\in (y^{-1},1)\cdot Z\}$$
	is a coset of a subgroup of $\G_m^n$. Therefore, $T=\Stab_{\G_m^n}(F_y)$ is a translation of $F_y$. Hence the general fiber $F_y$ of $\phi$ has dimension $\dim(F_y)=\dim(T)=0$. We conclude that $\phi$ is generically finite and $n=\dim(X)$.
\end{proof}
\begin{proof}[Proof of Corollary \ref{chequiv}]
	Fix an embedding $X\stackrel{\iota}{\hookrightarrow}\A^N$ over $\overline{\Q}$ ($N\in\Z_{>0}$).
	
	The equivalence \eqref{chequiv1}$\Leftrightarrow$\eqref{chequiv2} is just Theorem \ref{chdimequal}.
	
	The direction \eqref{chequiv3}$\Rightarrow$\eqref{chequiv1} follows from Theorem \ref{thmmain}.
	
	It suffices to prove the direction \eqref{chequiv2}$\Rightarrow$\eqref{chequiv3}. Assume that $(X,f)$ is of strongly monomial type. Over $\overline{\Q}$, there exist integers $l\geq1$, a group endomorphism $g:\G_{m}^d\to\G_{m}^d$, and a dominant morphism $\phi:\G_{m}^d\to X$ such that $f^{\circ l}\circ\phi=\phi\circ g$, where $d:=\dim(X)$. Take a number field $K$ such that $X,f,\iota,g,\phi$ are all defined over $K$. Since $(X,f)$ is cohomologically hyperbolic, so is $(X,f^{\circ l})$; hence the system $(\G_{m}^d,g)$ is also cohomologically hyperbolic because $\phi$ is generically finite. By Proposition \ref{Gmndyndeg}, $g$ is positive. By \cite[Proposition 21(d)]{Sil14}, the set of $g$-preperiodic points in $\G_m^d(\C)$ is $\PrePer(g)=\G_m^d(\overline{\Q})_\tors$. In particular,
	$$Q:=\PrePer(g)\cap\G_m^d(\overline{\Q})_\tors=\G_m^d(\overline{\Q})_\tors$$
	is Zariski-dense in $\G_n^d$. Observe that $Q\subseteq (U(\overline{\Q}))^d\subseteq\G_m^d(K^c)$. Then $Q\subseteq\G_m^d(K^c)$ satisfies the conditions (DCI) and (BH). Write $g=\varphi_A$, where $A\in M_d(\Z)$ is positive. Let $A^\adj$ be the adjoint matrix of $A$ and $\delta:=\det(A)\in\Z\setminus\{0\}$. From $\varphi_{A^\adj}\circ\varphi_{A}=\varphi_{\delta I_d}$, it is easy to see that $g^{-1}(Q)=Q=g(Q)$. We conclude that the subset $Q\subseteq\G_m^d(K^c)$ satisfies (DCI), (BH), and (AI) for the system $(\G_m^d,g)$. Let $P:=g(Q)\subseteq X(K^c)$, which is Zariski-dense in $X$ because $g$ is dominant. It is easy to see that $P$ satisfies (DCI) and (BH), since $g$ and $\iota$ are defined over $K$. By $Q=g(Q)$ and $f^{\circ l}\circ\phi=\phi\circ g$, we have $P\setminus (f^{\circ l})^{-1}(P)=\emptyset$; hence $P$ satisfies the condition (AI) for the system $(X,f^{\circ l})$. The proof is completed.
\end{proof}

\section{Proofs of the applications}\label{SApp}
\begin{proof}[Proof of Theorem \ref{appaff}]
	Using the very ample line bundle $\sO(1)$ on $\P^N$ and $f^*\sO(1)=\sO(1)^{\otimes d}$, it is easy to obtain $\la_i(f)=d^i$ for $0\leq i\leq N$. In particular, $f$ is cohomologically hyperbolic.
	 
	By Theorems \ref{thmmain} and \ref{chdimequal}, it suffices to show that the subset $P$ satisfies (DCI), (BH), and (AI).
	 
	It is clear that $P\subseteq f^{-1}(P)$, hence (AI) holds for $P$.
	 
	We consider (BH) for $P$. Fix a norm $\Vert\cdot\Vert$ on $\C^N$. Let
	\begin{equation}\label{Green}
	 G:\C^N\to\R_{\geq0},G(z)=\lim_{n\to\infty}\frac{1}{d^n}\log\max\{1,\Vert f^{\circ n}(z)\Vert\}
	\end{equation}
	be the Green function associated with $f$, which is a continuous plurisubharmonic function such that 
	\begin{equation}\label{Greenprop}
	 G(z)=\log\Vert z\Vert+O(1)\text{ as }\Vert z\Vert\to+\infty\text{, and }G(f(z))=d\cdot G(z).
	\end{equation}
	See \cite{BJ00} for more details on $G$. For every $f$-preperiodic point $x=(x_1,\cdots,x_N)$ in $\A^N(\C)$, the sequence
	$$\left(\log\max\{1,\Vert f^{\circ n}(x)\Vert\}\right)_{n\geq1}$$
	is bounded, so $x$ is contained in the compact set
	$$K(f):=\{z\in\C^N:G(z)=0\}\subset\C^N.$$
	We conclude that $\max\{\left|x_i\right|:1\leq i\leq N\}\leq R(f)$, for some constant $R(f)>0$ depending only on $f$. Set $R:=\max\{R(\sigma(f)):\sigma\in \Gal(\overline{\Q}/\Q))\}>0$. Then we see that $C(x)\leq R$ for every $x\in P$. Thus (BH) holds for $P$.
	 
	\medskip
	 
	We check the condition (DCI) for $P$. The set $P$ is Zariski-dense in $\A^N$ by assumption. For $1\leq i\leq N$, let $f_i^+$ be the sum of the monomials of degree $d$ in $f_i$. Since $f$ is a regular endomorphism, each $f_i^+$ is a non-zero homogeneous polynomial of degree $d$ in $z_1,\dots,z_N$. Set $h_i=f_i-f_i^+$, which has degree at most $d-1$. Let $f_h=(f_1^+,\dots,f_N^+):\A^N\to\A^N$ be the homogeneous part of $f$. Since $f$ is a regular endomorphism, we have $f_h ^{-1}(0)=\{0\}$ in $\A^N(\C)$. By Hilbert's Nullstellensatz, there exists an integer $m\geq d$ such that
	$$(z_1^m,\dots,z_N^m)\subseteq (f_1^+,\dots,f_N^+)$$
	as ideals of the polynomial ring $\overline{K}[z_1,\dots,z_N]$. Then we can take $(R_{ij})_{1\leq i,j\leq N}\subset \overline{K}[z_1,\dots,z_N]$ such that $z_i^m=\sum_{j=1}^N R_{ij}f_j^+$ for $1\leq i\leq N$, where each $R_{ij}$ is a homogeneous polynomial of degree $m-d$ in $z_1,\dots,z_N$ over $\overline{K}$. Since both the $z_i^m$'s and the $f_i^+$'s have coefficients in $K$, we may assume that each $R_{ij}$ has $K$-coefficients. For every finite place $v\in\sM_K$, fix an arbitrary extension of the absolute value $\left|\cdot \right|_v$ on $K$ (normalized in some way) to $\overline{\Q}$. Let $v\in\sM_K$ be a finite place. For a polynomial
	$$R=\sum_{I\in\Z_{\geq0}^N}a_I z^I\in\overline{K}[z_1,\dots,z_N],$$
	define $\left|R\right|_v=\max\{\left|a_I\right|_v:I\in\Z_{\geq0}^N\}$. For a point $z=(z_1,\dots,z_N)\in\A^N(\overline{K})$, define $\left|z\right|_v=\max\{\left|z_i\right|_v:1\leq i\leq N\} $. Set $B_v=\max\{\left|A_{ij}\right|_v:1\leq i,j\leq N\}>0$, $H_v=\max\{\left|h_{i}\right|_v:1\leq i\leq N\}$, and
	\begin{equation}\label{affAv}
	 A_v:=\max\{1,B_v H_v,B_v^{1/(d-1)}\}\geq1.
	\end{equation}
	We claim that
	\begin{equation}\label{affDCIv}
	 \text{for every } z\in\A^N(\overline{K})\text{ with }\left|z\right|_v>A_v\text{, we have }\left|f(z)\right|_v>\left| z\right| _v .
	\end{equation}
	Let $z\in\A^N(\overline{K})$ such that $\left|z\right|_v=\left|z_i\right|_v> A_v$ where $1\leq i\leq N$. Then
	\begin{align*}
	 \left|z\right|_v^m&=\left|z_i^m\right|_v=\left|\sum_{j=1}^N R_{ij}(z)f_j^+(z)\right|_v\\
	 &\leq \max_{1\leq j\leq N} \left|R_{ij}(z)\right|_v\cdot\left|f_i^+(z)\right|_v\leq\left( \max_{1\leq j\leq N} \left|R_{ij}(z)\right|_v\right) \left|f_h(z)\right|_v\\
	 &\leq B_v\left|z \right|_v^{m-d} \left|f_h(z)\right|_v,
	\end{align*}
	hence 
	\begin{equation*}
	 \left|f_h(z)\right|_v\geq B_v^{-1}\left|z\right|_v^d.
	\end{equation*}
	Take $1\leq i^\prime\leq N$ such that $\left|f_h(z)\right|_v=\left|f_{i^\prime}^+(z)\right|_v$. We have
	$$\left|f_{i^\prime}^+(z)\right|_v=\left|f_h(z)\right|_v\geq B_v^{-1}\left|z\right|_v^d> H_v\left|z\right|_v^{d-1}\geq \left|h_{i^\prime}(z)\right|_v,$$
	hence
	\begin{align*}
	 \left|f(z)\right|_v&\geq\left|f_{i^\prime}(z)\right|_v=\left|f_{i^\prime}^+(z)-h_{i^\prime}(z)\right|_v=\left|f_{i^\prime}^+(z)\right|_v\\
	 &\geq B_v^{-1}\left|z\right|_v^d>\left|z\right|_v,
	\end{align*}
	i.e., \eqref{affDCIv} holds. Inductively using \eqref{affDCIv}, we conclude that for every $z\in\A^N(\overline{K})$ with $\left|z\right|_v> A_v$, the sequence $(\left|f^{\circ n}(z)\right|_v)_{n\geq1}$ is strictly increasing, so $z$ cannot be $f$-preperiodic.
	 
	Since there are only finitely many non-archimedean $v\in\sM_K$ such that $A_v>1$, we can take an integer $M\geq1$ such that $\left|M\right|_v\leq A_v^{-1}$ for all finite $v\in\sM_K$. Let $z\in P$. For every finite $v\in\sM_K$, $\left|Mz\right|_v\leq A_v^{-1}A_v=1$ by \eqref{affDCIv}, since $z$ is $f$-preperiodic. Hence the coordinates $Mz_1,\dots,Mz_N$ are all algebraic integers. We conclude that (DCI) holds for $P$.
\end{proof}
\begin{proof}[Proof of Theorem \ref{back}]
	It suffices to prove the ``if'' direction. Assume that $P$ is Zariski-dense in $\A^N$. We show that $P$ satisfies (AI), (BH), and (DCI). After enlarging $K$ if necessary, we may assume $x\in \A^N(K)$.
	 
	It is clear that $P\setminus f^{-1}(P)\subseteq\{x\}$ is not Zariski-dense in $\A^N$, i.e., (AI) holds.
	
	We consider (BH) for $P$. Fix a norm $\Vert\cdot\Vert$ on $\C^N$. Similar to the proof of Theorem \ref{appaff}, let $G:\C^N\to\R_{\geq0}$ be the Green function associated with $f$, as given in \eqref{Green}. By \eqref{Greenprop}, every point $z\in P$ is contained in the compact set
	$$\{w\in\C^N:G(w)\leq G(x)\}\subset\C^N.$$
	Thus, there exists a constant $R(f,x)>0$ depending only on $f$ and $x$ such that $\max\{\left|z_i\right|:1\leq i\leq N\}\leq R(f,x)$ for all $z=(z_1,\dots,z_N)\in P$. Set $$R_x:=\max\{R(\sigma(f),\sigma(x)):\sigma\in \Gal(\overline{\Q}/\Q))\}>0.$$
	We see that $C(z)\leq R_x$ for all $z\in P$. Thus (BH) holds for $P$.
	
	We check (DCI) for $P$. The set $P$ is Zariski-dense in $\A^N$ by assumption. For every finite place $v\in\sM_K$, let $A_v\geq1$ be given as in \eqref{affAv} in the proof of Theorem \ref{appaff}, and set
	$$A_v(x):=\max\{A_v,\left|x\right|_v\}.$$
	Inductively using \eqref{affDCIv}, for every $z\in\A^N(\overline{K})$ with $\left|z\right|_v>A_v(x)$, we have
	$$\left|f^{\circ n}(z)\right|_v>\left|z\right|_v>A_v(x)\geq\left|x\right|_v$$
	for all $n\geq1$ (in particular, $x\notin O_f(z)$, the $f$-forward orbit of $z$). We conclude that for every $z\in P$ and every finite $v\in\sM_K$, $\left|z\right|_v\leq A_v(x)$. Since there are only finitely many non-archimedean $v\in\sM_K$ such that $A_v(x)>1$, we can take an integer $M_x\geq1$ such that $\left|M_x\right|_v\leq A_v(x)^{-1}$ for all finite $v\in\sM_K$. For every $z\in P$ and every finite $v\in\sM_K$, $\left|M_x z\right|_v\leq A_v(x)^{-1}A_v(x)=1$, so the coordinates $M_x z_1,\dots,M_x z_N$ are all algebraic integers. Thus (DCI) holds for $P$.
	
	From the proof of Theorem \ref{appaff}, $f$ is cohomologically hyperbolic (in fact, polarized), so the last statement follows directly from Theorems \ref{thmmain} and \ref{chdimequal}.
\end{proof}
\begin{proof}[Proof of Theorem \ref{Henon}]
	Let 
	$$d=\deg_1(f)\geq2,\quad d_{-}=\deg_1(f^{-1}),\quad p=\dim(I(f))+1,\quad q=\dim(I(f^{-1}))+1.$$
	By \cite[Proposition 2.3.2]{SibonyFrench}, we have $p+q=N$ and $d^q=d_{-}^p$. In particular, $d_{-}\geq2$. Dinh and Sibony \cite{DS} computed the dynamical degrees of $f$ as follows (see also \cite[\S 1.1]{Thelin}):
	$$\la_i(f)=d^i\text{ for }0\leq i\leq q\quad\text{and}\quad\la_j(f)=d_{-}^{N-j}\text{ for }q\leq j\leq N.$$
	Then $f$ is $q$-cohomologically hyperbolic.
	
	Assume that $P=\Per(f,\A^N(K^c))$ is Zariski-dense in $\A^N$. We will deduce a contradiction.
	
	We show that the subset $P$ satisfies (DCI), (BH), and (AI). 
	
	By \cite[Lemma 6.1]{Kawaguchi}, the inverse $f^{-1}$ is also defined over $K\subseteq K^c$, so $P=f^{-1}(P)$ and (AI) holds for $P$.
	
	For every place $v\in\sM_K$, fix an arbitrary extension of the absolute value $\left|\cdot \right|_v$ on $K$ (normalized in some way) to $\overline{K_v}$, where $\overline{K_v}$ is the algebraic closure of the completion $K_v$ at $v$. Fix an arbitrary $v\in\sM_K$. For $z=(z_1,\cdots,z_N)\in\A^N(\overline{K_v})$, set $\left|z\right|_v:=\max\{\left|z_i\right|_v:1\leq i\leq N\}$. Define two non-negative functions $G_{f,v},G_{f^{-1},v}:\A^N(\overline{K_v})\to\R$ by
	\begin{align*}
		G_{f,v}(z)=&\lim\limits_{n\to\infty}\frac{1}{d^n}\log\max\{\left|f^{\circ n}(z) \right|_v ,1\},\\
		G_{f^{-1},v}(z)=&\lim\limits_{n\to\infty}\frac{1}{d_{-}^n}\log\max\{\left|(f^{-1})^{\circ n}(z)\right|_v ,1\}.
	\end{align*}
	These limits in the above definition exist \cite{SibonyFrench,Kawaguchi}. By Kawaguchi \cite[Theorems A(3) and Theorem 5.1]{Kawaguchi}, there exist subsets $V^+_v,V^-_v\subseteq\A^N(\overline{K_v})$ with 
	\begin{equation}\label{unionall}
		V^+_v\cup V^-_v=\A^N(\overline{K_v})
	\end{equation}
	and constants $c^+_v,c^-_v\in\R$ such that
	\begin{align}
		\label{V+ineq}G_{f,v}(\cdot)&\geq \log\max\{\left|\cdot\right|_v,1\}+c^+_v\text{ on }V^+_v;\\
		\label{V-ineq}G_{f^{-1},v}(\cdot)&\geq \log\max\{\left|\cdot\right|_v,1\}+c^-_v\text{ on }V^-_v.
	\end{align}
	Moreover, Kawaguchi showed that \cite[Theorems B(1)]{Kawaguchi} one can require 
	\begin{equation}
		\label{almost1}c^+_v=c^-_v=0
	\end{equation}
	for all but finitely many $v\in\sM_K$.
	Let $x\in P$ be a $K^c$-rational $f$-periodic point. Since $d,d_-\geq2$, by definition we obtain \begin{equation}\label{perG0}
		G_{f,v}(x)=G_{f^{-1},v}(x)=0.
	\end{equation}
	Set 
	\begin{equation}
		\label{Av}A_v=\max\{1,\exp(-c^+_v),\exp(-c^-_v)\}\geq1.
	\end{equation}
	Then $\left|x\right|_v\leq A_v$ by \eqref{unionall}-\eqref{V-ineq} and \eqref{perG0}.
	
	We conclude that for every $x\in P$, $C(x)\leq c$, where $c\in\R_{\geq1}$ is the maximum of all $A_v$ for archimedean $v\in\sM_K$. Hence $P$ satisfies (BH).
	
	By \eqref{almost1} and \eqref{Av}, the upper bound $A_v=1$ for all but finitely many $v\in\sM_K$. Thus we can take an integer $M\geq1$ such that $\left|M\right|_v\leq A_v^{-1}$ for all finite places $v\in\sM_K$. For every $x=(x_1,\dots,x_N)\in P$ and every finite place $v\in\sM_K$, $$\left|Mx_i\right|_v\leq\left|M\right|_v\left|x\right|_v\leq A_v^{-1}A_v=1,$$
	so $Mx_i$ is an algebraic integer ($1\leq i\leq N$). Thus, $P$ satisfies (DCI).
	
	By Theorems \ref{thmmain} and \ref{chdimequal}, we see that $f$ is of strongly monomial type. We make a base change and work over $k=\overline{\Q}$. Then there exist an integer $l\geq1$, a group endomorphism $g=\varphi_A:\G_m^N\to\G_m^N$, and a dominant morphism $\phi:\G_m^N\to \A^N$ such that $f^{\circ l}\circ\phi=\phi\circ g$, where $A\in M_N(\Z)$ is a matrix with $\det(A)\neq0$. Since the iterate $f^{\circ l}$ is still an automorphism of H\'enon type \cite[Theorem 7.10(a)]{241}, we may assume $l=1$. Since $\phi:\G_m^N\to \A^N$ is generically finite, $\varphi_A$ and $f$ have the same dynamical degrees. In particular, by Lemma \ref{Gmndyndeg}, we have $$\left|\det(A)\right|=\la_N(g)=\la_N(f)=1,$$
	so $\det(A)=\pm1$ and $A\in\GL_N(\Z)$. Let $\nu_1,\dots,\nu_N$ be the eigenvalues of $A$ in $\C$ (counted with multiplicity) such that $\left| \nu_1\right| \geq\cdots\geq\left| \nu_N\right|>0$. Since $\nu_1$ is an eigenvalue of the matrix $A\in\GL_N(\Z)$, it is an algebraic unit (i.e., both $\nu_1$ and $\nu_1^{-1}$ are algebraic integers). Hence its complex conjugate $\overline{\nu_1}$ is also an algebraic unit. However, by Lemma \ref{Gmndyndeg}, we deduce that
	$$d^2=\la_1(f)^2=\la_1(g)^2=\left|\nu_1\right|^2=\nu_1 \overline{\nu_1}$$
	is also an algebraic unit, which is a contradiction because $d^2$ is an integer at least $4$. We conclude that $P$ is not Zariski-dense in $\A^N$.
\end{proof}

\subsection*{Acknowledgement}
The first-named author Zhuchao Ji is supported by National Key R\&D Program of China (No.2025YFA1018300), NSFC Grant (No.1240\\1106), and ZPNSF grant (No.XHD24A0201). The second and third-named authors Junyi Xie and Geng-Rui Zhang are supported by NSFC Grant (No.12271007). The authors thank Joseph Silverman for suggesting that we consider automorphisms of H\'enon types on $\A^N$ for all $N\geq2$. The authors would like to thank Umberto Zannier, Sheng Meng, Qi Zhou, and Jason Bell for helpful comments on the first version of this paper.

\end{document}